\documentclass[reqno]{amsart}
\usepackage{amssymb}
\usepackage{setspace}
\usepackage{hyperref}
\newtheorem{theorem}{Theorem}[section]

\newtheorem{corollary}[theorem]{Corollary}

\newtheorem{identity}[theorem]{Identity}

\numberwithin{equation}{section}

\begin{document}
%%%%%%%%%%%%%%%%%%%%%%%%%%
\title[]{Algebraic identities on q-harmonic numbers and q-binomial coefficients}
\author{ Said Zriaa and Mohammed Mou\c{c}ouf}
\address{ Said Zriaa and Mohammed Mou\c{c}ouf,
University Chouaib Doukkali.
Department of Mathematics, Faculty of science
Eljadida, Morocco}
\email{saidzriaa1992@gmail.com}
\email{moucou@hotmail.com}

\begin{abstract}
The aim of this paper is to present a general algebraic identity. Applying this identity,
we provide several formulas involving the q-binomial coefficients and the q-harmonic numbers.
We also recover some known identities including an algebraic identity of D. Y. Zheng on q-Ap\'{e}ry numbers and we establish
the q-analog of Euler's formula.
The proposed results may have important applications in the theory of q-supercongruences.

\end{abstract}
%-----------------------------------
\maketitle	
\section{Introduction}
For an indeterminate $x$ the q-shifted factorial is usually defined by 
\begin{equation*}
(x;q)_{0}=1 \,\,\ \mbox{and} \,\,\ (x;q)_{n}=(1-x)(1-qx)\cdots(1-q^{n-1}x) \,\,\  \mbox{for} \,\,\  n=1,2,\ldots
\end{equation*}
The Gaussian q-binomial coefficient is correspondingly given by
\begin{equation*}
{n \brack j}=\frac{(q;q)_{n}}{(q;q)_{j}(q;q)_{n-j}} \,\,\ \mbox{for} \,\,\  j=0,1,2,\ldots,n,
\end{equation*}
\\
The q-harmonic numbers are defined by 
\begin{equation*}
\mathcal{H}_{n}(q)=\sum_{k=1}^{n}\frac{1}{[k]_{q}} \,\,\ \mbox{for} \,\,\  n=1,2,\ldots,
\end{equation*}
where the q-numbers are given by
\begin{equation*}
[k] \,\,\ \mbox{or} \,\,\ [k]_{q}:=\frac{1-q^{k}}{1-q}=1+q+\cdots+q^{k-1}
\end{equation*}
Following Comtet~\cite{Comtet}, the complete Bell polynomials can be explicitly expressed as
\begin{equation*}
\mathbf{B}_{n}(x_{1},x_{2},\cdots,x_{n})=\sum_{m_{1}+2m_{2}+\cdots+nm_{n}=n}\frac{n!}{m_{1}!m_{2}!\cdots m_{n}!}\bigg(\frac{x_{1}}{1!}\bigg)^{m_{1}}\bigg(\frac{x_{2}}{2!}\bigg)^{m_{2}}\cdots \bigg(\frac{x_{n}}{n!}\bigg)^{m_{n}}
\end{equation*}

During the last two decades, there has been an increasing interest in studying binomial sums and their q-analogues,
one can consult recent papers, for example~\cite{Wchu7,Xwang,Qyan}. For a comprehensive account of the
q-series and its applications to numbers theory, combinatorics and special functions, we refer the reader to the excellent
monograph by G. Gasper and M. Rahman \cite{Ggasper}. Recently, the study of the q-harmonic congruences turned into a very active research area (see e.g., \cite{Astraub} and the references therein).

There exist numerous combinatorial identities involving q-binomial coefficients in the mathematical literature (see e.g.,~\cite{Wchu7,Ekilic,Jxsharon,Xwang,Qyan}). 
Nowadays, there has been growing interest in deriving q-analogues of several combinatorial identities. This 
includes, for example, q-generalizations of some well known identities involving harmonic numbers.\\

%%%%%%%%%%%%%%%%%%%%%%%%%%%%%%%%%%%%%%%%%%%%%
The Ap\'{e}ry numbers are defined by the following binomial sum
\begin{equation*}
A(n)=\sum_{k=0}^{n}\binom{n}{k}^{2}\binom{n+k}{k}^{2}
\end{equation*}
These numbers have many interesting properties that make them extremely useful in the proof of the irrationality of $\zeta(3)$ (see~\cite{Rapery} for further details). Also the Ap\'{e}ry numbers have remarkable arithmetic properties~\cite{Igessel}. F. Beukers conjectured~\cite{Beukers} that
\begin{equation*}
A\bigg(\frac{p-1}{2}\bigg)\equiv a(p)\,\,\ (mod p^{2})
\end{equation*}
where $p$ is an odd prime and $a(n)$ is determined by 
\begin{equation*}
\sum_{k=1}^{\infty}a(k)q^{k}:=q\prod_{k=1}^{\infty}(1-q^{2k})^{4}(1-q^{4k})^{4}=q-4q^{3}-2q^{5}+24q^{7}+\cdots.
\end{equation*}
Beukers' conjecture was later showed by S. Ahlgren, K. Ono~\cite{Sahlgren}, who reduce this statement in terms of the harmonic numbers $H_{n}$ to the identity 
\begin{equation*}
\sum_{k=1}^{n}\binom{n}{k}^{2}\binom{n+k}{k}^{2}\bigg(1+2kH_{n+k}+2kH_{n-k}-4kH_{k} \bigg)=0
\end{equation*}
In order to give a classical proof of the last identity, Chu~\cite{Wchu} presented the following algebraic identity
\begin{align*}
\sum_{j=0}^{n}\binom{n}{j}^{2}\binom{n+j}{j}^{2}\bigg(\frac{-j}{(x+j)^{2}}+\frac{1+2j(H_{n+j}-H_{j})+2j(H_{n-j}-H_{j})}{x+j} \bigg)&=\frac{x(1-x)_{n}^{2}}{(x)_{n+1}^{2}}
\end{align*}
which gives the desired formula in the limit.\\
%%%%%%%%%%%%%%%%%%%%%%%%%%%%%%%%%%%%%%%%%%%%
\indent In order to prove irrationality results on the q-analog of $\zeta(3)$:
\begin{equation*}
\zeta_{q}(3)=\sum_{k=1}^{+\infty}\frac{q^{k}(1+q^{k})}{(1-q^{k})^{3}}  
\end{equation*}
The authors of~\cite{Ckrattenthaler} introduced implicitly a q-analog of the Ap\'{e}ry numbers $A_{q}^{KRZ}(n)$
and they showed that
\begin{equation*}
A_{q}^{KRZ}(n)=\sum_{k=0}^{n}\frac{a_{q}(n,k)}{q^{k}}
\end{equation*}
where $a_{q}(n,k)$ can be defined via the following q-partial fraction decomposition
\begin{equation}\label{eq: Ckra}
\frac{(xq^{-n};q)_{n}^{2}}{(x;q)_{n+1}^{2}}=\sum_{k=0}^{n}\bigg(\frac{a_{q}(n,k)}{(1-q^{k}x)^{2}}+\frac{b_{q}(n,k)}{1-q^{k}x} \bigg)
\end{equation} 
A. Straub~\cite{Astraub} showed via the partial fraction decomposition technique that $A_{q}^{KRZ}(n)$ have the following explicit q-binomial representation
\begin{equation*}
q^{n(2n+1)}A_{q}^{KRZ}(n)=\sum_{k=0}^{n}q^{(n-k)^{2}}{n \brack k}^{2}{n+k\brack k}^{2}
\end{equation*} 
which reduces to Ap\'{e}ry numbers when $q\rightarrow1$.\\
\indent D. Y. Zheng~\cite{Dyzheng} has recently introduced the q-Ap\'{e}ry numbers
\begin{equation*}
q^{n(n+1)}A_{q}^{KRZ}(n)=\sum_{k=0}^{n}q^{k(k-2n)}{n \brack k}^{2}{n+k\brack k}^{2}
\end{equation*}
then he established an interesting algebraic identity
{\small
\begin{equation}\label{eq: Zheng}
\frac{x^{2n}(q/x;q)_{n}^{2}}{(1-x)(xq;q)_{n}^{2}}=\frac{1}{1-x}+\sum_{j=1}^{n}{n \brack j}^{2}q^{j(j-2n)}\bigg(\frac{q^{j}-1}{(1-xq^{j})^{2}}+\frac{1-4[j]\mathcal{H}_{j}(q)+2[j]\mathcal{H}_{n+j}(q)+2q[j]\mathcal{H}_{n-j}(q^{-1})}{1-xq^{j}}\bigg)
\end{equation}
}
which is a q-extension of the Chu's identity, and obtained the identity
\begin{equation*}
\sum_{k=0}^{n}q^{k(k-2n)}{n \brack k}^{2}{n+k\brack k}^{2}\bigg(2\mathcal{H}_{k}(q)-\mathcal{H}_{n+k}(q)-q\mathcal{H}_{n-k}(q^{-1}) \bigg)=0
\end{equation*}
%%%%%%%%%%%%%%%%%%%%%%%%%%%%%%%%%%%%%%%%%%%
The purpose of this paper is to establish and develop important algebraic identities involving q-harmonic numbers and q-binomial coefficients, which may have important applications in the theory of q-supercongruences.\\
%%%%%%%%%%%%%%%%%%%%%%%%%%%%%%%%%%%%%%%%%%%%%%%%%%%%%%%%%%%%%%%%%%%%%%
\section{Some identities of q-binomial coefficients}
We now state and prove one of the main results of this paper.
\begin{theorem}\label{thm: ccc}
Let $\alpha_{1},\alpha_{2},\ldots,\alpha_{s}$ be distinct elements of $\mathbb{C}$. For a positive integer $m$ let $P(x)=(x-\alpha_{1})^{m}(x-\alpha_{2})^{m}\cdots(x-\alpha_{s})^{m}$. For any polynomial $Q(x)$ such that $\deg(Q)<\deg(P)$, we have
\begin{equation}\label{eq: eee}
\frac{Q(x)}{P(x)}=\sum_{j=1}^{s}\sum_{i=0}^{m-1}\sum_{k=0}^{m-1-i}\frac{(-1)^{k}g_{j}(\alpha_{j})\mathbf{B}_{k}(x_{1},\cdots,x_{k})Q^{(i)}(\alpha_{j})}{i!k!(x-\alpha_{j})^{m-i-k}}.
\end{equation}
where
\begin{equation*}
x_{l}=m(l-1)!\sum_{i=1,i\neq j}^{s}\frac{1}{(\alpha_{j}-\alpha_{i})^{l}} \,\,\ \mbox{and} \,\,\ g_{j}(x)=\prod_{i=1,i\neq j}^{s}(x-\alpha_{i})^{-m}.
\end{equation*}
\end{theorem}
\begin{proof}
Following~\cite[Eq.4]{Mmous}, we can write
\begin{equation*}
Q(x)=\sum_{j=1}^{s}\sum_{i=0}^{m-1}\frac{1}{i!}Q^{(i)}(\alpha_{j})L_{ji}(x)[P].
\end{equation*}
By virtue of Equation $(2)$ of~\cite{Mmous}, it is clear that
\begin{equation*}
L_{ji}(x)[P]=P(x)\sum_{k=0}^{m-1-i}\frac{g^{(k)}_{j}(\alpha_{j})}{k!(x-\alpha_{j})^{m-i-k}}
\end{equation*}
If we combine the two previous identities, we get
\begin{equation*}
\frac{Q(x)}{P(x)}=\sum_{j=1}^{s}\sum_{i=0}^{m-1}\sum_{k=0}^{m-1-i}\frac{g_{j}^{(k)}(\alpha_{j})Q^{(i)}(\alpha_{j})}{i!k!(x-\alpha_{j})^{m-i-k}}.
\end{equation*}
Since
\begin{equation*}
g_{j}(x)=\phi(x)\circ f_{j}(x)
\end{equation*}
where $\phi(x)=\exp(mx)$ and $f_{j}(x)=\ln(\prod_{i=1,i\neq j}^{s}(x-\alpha_{i})^{-1})$. Then $\phi^{(k)}(x)=m^{k}\exp(mx)$ and $f^{(k)}_{j}(x)=(-1)^{k}(k-1)!\mathcal{H}_{k,\mathcal{\alpha}_{s}[j]}(x)$, where $\mathcal{H}_{l,\mathcal{\alpha}_{s}[j]}(x)=\sum_{i=1,i\neq j}^{s}\frac{1}{(x-\alpha_{i})^{l}}$. By using the Fa\`{a} di Bruno formula, we can easily prove that
\begin{equation*}
g_{j}^{(k)}(x)=(-1)^{k}g_{j}(x)\sum_{m_{1}+2m_{2}+\cdots+km_{k}=k}\frac{k!}{m_{1}!m_{2}!\cdots m_{k}!}\prod_{l=1}^{k}\bigg(\frac{m(l-1)!\mathcal{H}_{l,\mathcal{\alpha}_{s}[j]}(x)}{l!}\bigg)^{m_{l}}
\end{equation*}
In particular
\begin{equation*}
g_{j}^{(k)}(\alpha_{j})=(-1)^{k}g_{j}(\alpha_{j})\mathbf{B}_{k}(x_{1},\cdots,x_{k})
\end{equation*}
This gives the required result.
\end{proof}
Taking $\alpha_{i}=q^{-i}$ in Theorem~\eqref{thm: ccc}, we obtain after some minor manipulations the following theorem.
\begin{theorem}\label{thm: bbbb}
Let $m$ and $n$ positive integers and let $Q(x)$ be a polynomial such that $\deg(Q)<(n+1)m$, we have
\begin{equation*}
\frac{(q;q)_{n}^{m}Q(x)}{(x;q)_{n+1}^{m}}=\sum_{j=0}^{n}{n \brack j}^{m}q^{m\binom{j+1}{2}}\sum_{i=0}^{m-1}\sum_{k=0}^{m-1-i}\frac{(-1)^{mj+i}\mathbf{B}_{k}(x_{1},x_{2},\cdots,x_{k})Q^{(i)}(q^{-j})}{i!k!q^{j(i+k)}(1-xq^{j})^{m-i-k}}.
\end{equation*}
where
\begin{equation*}
x_{1}=\frac{q^{j}m}{1-q}\bigg(\mathcal{H}_{j}(q)-q\mathcal{H}_{n-j}(q^{-1})\bigg)
\end{equation*}
and
\begin{equation*}
x_{l}=m(l-1)!\sum_{i=0,i\neq j}^{n}\frac{q^{jl}}{(1-q^{j-i})^{l}}  \,\,\ \mbox{for} \,\,\  l=1,2,\ldots,m-1,
\end{equation*}
\end{theorem}
In view of Theorem~\eqref{thm: bbbb}, we establish interesting corollaries
\begin{corollary}\label{cor: algebaric}
Let $n$ be a positive integer and let $Q(x)$ be a polynomial such that $\deg(Q)<2(n+1)$. We have
\begin{equation*} 
\frac{(q;q)_{n}^{2}Q(x)}{(x;q)_{n+1}^{2}}=\sum_{j=0}^{n}{n \brack j}^{2}q^{j(j+1)}\bigg(\frac{Q(q^{-j})}{(1-xq^{j})^{2}}-\frac{q^{-j}Q^{'}(q^{-j})}{(1-xq^{j})}+\frac{2Q(q^{-j})}{(1-xq^{j})}\bigg(\frac{\mathcal{H}_{j}(q)-q\mathcal{H}_{n-j}(q^{-1})}{1-q}\bigg)\bigg)
\end{equation*}
\end{corollary}
\begin{identity}
Setting $Q(x)=1$ in the last corollary, we obtain
\begin{equation*} 
\frac{(q;q)_{n}^{2}}{(x;q)_{n+1}^{2}}=\sum_{j=0}^{n}{n \brack j}^{2}q^{j(j+1)}\bigg(\frac{1}{(1-xq^{j})^{2}}+\frac{2}{1-xq^{j}}\bigg(\frac{\mathcal{H}_{j}(q)-q\mathcal{H}_{n-j}(q^{-1})}{1-q}\bigg)\bigg)
\end{equation*}
In particular
\begin{equation*} 
(q;q)_{n}^{2}=\sum_{j=0}^{n}{n \brack j}^{2}q^{j(j+1)}\bigg(1+2\bigg(\frac{\mathcal{H}_{j}(q)-q\mathcal{H}_{n-j}(q^{-1})}{1-q}\bigg)\bigg)
\end{equation*}
\end{identity}
%%%%%%%%%%%%%%%%%%%%%%%%%%%%%%%%%%%%%%%%%%%
\begin{identity}
\begin{equation*}
\frac{x^{2n}(q/x;q)_{n}^{2}}{(1-x)(xq;q)_{n}^{2}}=\frac{1}{1-x}+\sum_{j=1}^{n}{n \brack j}^{2}q^{j(j-2n)}\bigg(\frac{q^{j}-1}{(1-xq^{j})^{2}}+\frac{1-4[j]\mathcal{H}_{j}(q)+2[j]\mathcal{H}_{n+j}(q)+2q[j]\mathcal{H}_{n-j}(q^{-1})}{1-xq^{j}}\bigg)
\end{equation*}
This identity recovers Zheng identity~\eqref{eq: Zheng}.
\end{identity}
\begin{proof}
Let 
\begin{equation*}
Q(x)=(1-x)(x-q)^{2}\cdots(x-q^{n})^{2}=(1-x)x^{2n}(q/x;q)_{n}^{2}
\end{equation*}
It is not difficult to verify that
\begin{equation*}
\frac{Q(x)}{(x;q)_{n+1}^{2}}=\frac{x^{2n}(q/x;q)_{n}^{2}}{(1-x)(xq;q)_{n}^{2}} \,\,\ , \,\,\  Q(q^{-j})=q^{-j(2n+1)}(q^{j}-1)(q;q)_{n}^{2}{n+j \brack j}^{2}
\end{equation*}
and
\begin{align*}
Q^{'}(q^{-j})=q^{-2nj}(q;q)_{n}^{2}{n+j \brack j}^{2}\bigg(-1-2[j]\bigg(\mathcal{H}_{n+j}(q)-\mathcal{H}_{j}(q)\bigg) \bigg)\\
\end{align*}
Applying Corollary~\eqref{cor: algebaric}, we get after some simplifications the desired identity.
\end{proof}
\begin{identity}
\begin{equation*}
\frac{(xq^{-n};q)_{n}^{2}}{(x;q)_{n+1}^{2}}=\sum_{j=0}^{n}{n \brack j}^{2}{n+j \brack j}^{2}q^{j(j+1)-n(n+2j+1)}\bigg(\frac{1}{(1-xq^{j})^{2}}+\frac{4\mathcal{H}_{j}(q)-2q\mathcal{H}_{n-j}(q^{-1})-2\mathcal{H}_{n+j}(q)}{(1-q)(1-xq^{j})}\bigg)
\end{equation*}
This identity gives the explicit representation of~\eqref{eq: Ckra}. 
\end{identity}
\begin{proof}
Let $Q(x)=(xq^{-n};q)_{n}^{2}$. We have $Q(q^{-j})=q^{-n(n+2j+1)}{n+j \brack j}^{2}$ and 
\begin{equation*}
Q^{'}(q^{-j})=2q^{-n(n+2j+1)+j}{n+j \brack j}^{2}\bigg(\frac{\mathcal{H}_{n+j}(q)-\mathcal{H}_{j}(q)}{1-q} \bigg)
\end{equation*}
Using Corollary~\eqref{cor: algebaric}, the identity follows.
\end{proof}
\begin{identity}
Letting $x=0$ in the last identity, we obtain
\begin{equation*}
\sum_{j=0}^{n}{n \brack j}^{2}{n+j \brack j}^{2}q^{j(j+1)-2nj}\bigg(1-q+4\mathcal{H}_{j}(q)-2q\mathcal{H}_{n-j}(q^{-1})-2\mathcal{H}_{n+j}(q)\bigg)=q^{n(n+1)}(1-q)
\end{equation*}
\end{identity}
%%%%%%%%%%%%%%%%%%%%%%%%%%%%%%%%%%%%%%%%%%%
\begin{corollary}
Let $m, n, l$ be tree positive integers such that $0\leq l<(n+1)m$. Then we have 
\begin{equation*}
\frac{(q;q)_{n}^{m}x^{l}}{(x;q)_{n+1}^{m}}=\sum_{j=0}^{n}{n \brack j}^{m}q^{m\binom{j+1}{2}}\sum_{i=0}^{m-1}\sum_{k=0}^{m-1-i}\binom{l}{i}\frac{(-1)^{mj+i}\mathbf{B}_{k}(x_{1},x_{2},\cdots,x_{k})}{k!q^{j(k+l)}(1-xq^{j})^{m-i-k}}.
\end{equation*}
where
\begin{equation*}
x_{l}=m(l-1)!\sum_{i=0,i\neq j}^{n}\frac{q^{jl}}{(1-q^{j-i})^{l}}
\end{equation*}
\end{corollary}
\begin{corollary}
Let $m$ and $n$ be positive integers. We have
\begin{equation*}
\frac{(q;q)_{n}^{m}}{(x;q)_{n+1}^{m}}=\sum_{j=0}^{n}{n \brack j}^{m}q^{m\binom{j+1}{2}}\sum_{k=0}^{m-1}\frac{(-1)^{mj}\mathbf{B}_{k}(x_{1},x_{2},\cdots,x_{k})}{k!q^{jk}(1-xq^{j})^{m-k}}.
\end{equation*}
where
\begin{equation*}
x_{l}=m(l-1)!\sum_{i=0,i\neq j}^{n}\frac{q^{jl}}{(1-q^{j-i})^{l}}
\end{equation*}

\end{corollary}
\begin{corollary}\label{cor: SSS}
Let $Q(x)$ be a polynomial such that $\deg(Q)<n+1$. Then we have
\begin{equation*}
\frac{(q;q)_{n}Q(x)}{(x;q)_{n+1}}=\sum_{j=0}^{n}(-1)^{j}{n \brack j }q^{\binom{j+1}{2}}\frac{Q(q^{-j})}{(1-xq^{j})}.
\end{equation*}
\end{corollary}
The limiting case of Theorem~\eqref{thm: bbbb}, is the following
\begin{theorem}\label{thm: aaaa}
Let $m$ and $n$ be two positive integers. Let $Q(x)$ be a polynomial of degree $l$ with leading coefficient $a_{l}$. Then we have the following curious identity:
\begin{align*}
\sum_{j=0}^{n}{n \brack j}^{m}q^{m\binom{j}{2}}\sum_{i=0}^{m-1}\frac{(-1)^{mj+i+1}\mathbf{B}_{m-1-i}(x_{1},x_{2},\cdots,x_{m-1-i})Q^{(i)}(q^{-j})}{i!(m-1-i)!}=\\ \displaystyle\begin{cases}
0 &\text{if}\quad 0\leq l< m(n+1)-1 ,\\
(-1)^{m(n+1)}(q;q)_{n}^{m}q^{-m\binom{n+1}{2}}a_{l} &\text{if}\quad l=m(n+1)-1.
\end{cases}
\end{align*}
where
\begin{equation*}
x_{l}=m(l-1)!\sum_{i=0,i\neq j}^{n}\frac{q^{jl}}{(1-q^{j-i})^{l}}
\end{equation*}
\end{theorem}
When $m=1$, the formula of Theorem~\eqref{thm: aaaa} reads explicitly as
\begin{corollary}
Let $n$ be a positive integer and $Q(x)$ be a polynomial of degree $l$ with leading coefficient $a_{l}$. Then we have the following identity:
\begin{align*}
\sum_{j=0}^{n}{n \brack j}q^{\binom{j}{2}}(-1)^{j}Q(q^{-j})=\displaystyle\begin{cases}
0 &\text{if}\quad 0\leq l< n ,\\
(-1)^{n}(q;q)_{n}q^{-\binom{n+1}{2}}a_{l} &\text{if}\quad l=n.
\end{cases}
\end{align*}
\end{corollary}
\begin{identity}
Letting $Q(x)=(x-1)^{l}$ in the last corollary, we obtain, after some simple calculations, that
\begin{align*}
\sum_{j=0}^{n}{n \brack j}q^{\binom{j}{2}-jl}(-1)^{j}[j]^{l}=\displaystyle\begin{cases}
0 &\text{if}\quad 0\leq l< n ,\\
(-1)^{n}\dfrac{(q;q)_{n}}{(1-q)^{n}} q^{-\binom{n+1}{2}} &\text{if}\quad l=n.
\end{cases}
\end{align*}
We remark that in the limiting case $q\rightarrow 1$ of this identity, we obtain the famous Euler's formula\cite{Halzer,Hwgould2,Eakaratsuba,Cphoata}. Therefore,
this identity is the q-analog of Euler's formula:
\begin{align*}
\sum_{j=0}^{n}\binom{n}{j}(-1)^{j}j^{l}=\displaystyle\begin{cases}
0 &\text{if}\quad 0\leq l< n ,\\
(-1)^{n}n! &\text{if}\quad l=n.
\end{cases}
\end{align*}
\end{identity}
Opting $m=2$ in Theorem~\eqref{thm: aaaa}, we gain the following important result
\begin{corollary}\label{cor: dsdf}
Let $n$ be a positive integer and $Q(x)$ be a polynomial of degree $l$ with leading coefficient $a_{l}$. Then the following identity holds:
\begin{align*}
\sum_{j=0}^{n}{n \brack j}^{2}q^{j(j-1)}\bigg(Q^{'}(q^{-j})-\frac{2q^{j}Q(q^{-j})}{1-q}\bigg(\mathcal{H}_{j}(q)-q\mathcal{H}_{n-j}(q^{-1})\bigg)\bigg)=\displaystyle\begin{cases}
0 &\text{if}\quad 0\leq l< 2n+1 ,\\
(q;q)_{n}^{2}q^{-n(n+1)}a_{l} &\text{if}\quad l=2n+1.
\end{cases}
\end{align*}
\end{corollary}
\begin{identity}
If $Q(x)=1$, then the formula of Corollary~\eqref{cor: dsdf} gives
\begin{equation*}
\sum_{j=0}^{n}{n \brack j}^{2}q^{j^{2}}\bigg(\mathcal{H}_{j}(q)-q\mathcal{H}_{n-j}(q^{-1})\bigg)=0
\end{equation*}
\end{identity}
\begin{identity}
Choosing $Q(x)=(1-x)(x-q)^{2}\cdots(x-q^{n})^{2}$ in Corollary~\eqref{cor: dsdf}, we obtain
\begin{equation*}
\sum_{j=0}^{n}{n \brack j}^{2}{n+j \brack j}^{2}q^{j(j-1)-2nj}\bigg(1+2[j]\mathcal{H}_{n+j}(q)-4[j]\mathcal{H}_{j}(q)+2q\mathcal{H}_{n-j}(q^{-1})\bigg)\bigg)=q^{-n(n+1)}
\end{equation*}
\end{identity}
It would be interesting to establish some identities by means of Corollary~\eqref{cor: SSS}.
%%%%%%%%%%%%%%%%%%%%%%%%%%%%%%%%%%%%%%%%%%%%%%%%%%%%%%%%%%%%%%%%%%%%%
\begin{theorem}
Let $n$ be a positive integer and $y$ any complex number. Then the following identity holds true:
\begin{equation*}
\sum_{j=0}^{n}(-1)^{j}{n \brack j }q^{\binom{j+1}{2}-jn}\frac{(y;q)_{n+j}}{(y;q)_{j}(1-xq^{j})}=\frac{(q;q)_{n}(y/x;q)_{n}x^{n}}{(x;q)_{n+1}}.
\end{equation*}
\end{theorem}
\begin{proof}
Let $Q(x)$ be the polynomial 
\begin{equation*}
Q(x)=(x-y)(x-qy)\cdots(x-yq^{n-1})=(y/x;q)_{n}x^{n}
\end{equation*}
Since
\begin{equation*}
Q(q^{-j})=(yq^{j};q)_{n}q^{-jn} \,\,\ \mbox{and} \,\,\ (yq^{j};q)_{n}=\frac{(y;q)_{n+j}}{(y;q)_{j}}
\end{equation*}
we get form Corollary~\eqref{cor: SSS} the desired formula.
\end{proof}
%%%%%%%%%%%%%%%%%%%%%%%%%%%%%%%%%%%%%%%%%%%%%%%%%%%%%%%%%%%%%%%%%%%%%
Letting $y=q$ in the last theorem, we obtain the following result
\begin{theorem}\label{th: jjj}
Let $n$ be a positive integer, the following identities hold true
\begin{enumerate}
\item
\begin{equation*}
\sum_{j=0}^{n}(-1)^{j}{n \brack j}{n+j \brack j}\frac{q^{\binom{j+1}{2}-jn}}{(1-xq^{j})}=\frac{x^{n}(q/x;q)_{n}}{(x;q)_{n+1}}.
\end{equation*}
\item 
\begin{equation*}
\sum_{j=0}^{n}(-1)^{j}{n \brack j}{n+j \brack j}\frac{q^{\binom{j+1}{2}-jn}}{(1-q^{j+l})}=0.
\end{equation*}
for $ l=1,2,\ldots,n.$
\item 
\begin{equation*}
\sum_{j=0}^{n-1}(-1)^{n+1+j}{n \brack j}{n+j \brack j}\frac{(1-q^{n+l})}{(1-q^{j+l})}q^{\binom{j+1}{2}+n^{2}-jn}={2n \brack n}q^{\binom{n+1}{2}}.
\end{equation*}
for $ l=1,2,\ldots,n.$
\item 
\begin{equation*}
q^{\binom{n+1}{2}}=\sum_{j=0}^{n}(-1)^{n-j}{n \brack j}{n+j \brack j}q^{\binom{j+1}{2}-jn}.
\end{equation*}
\end{enumerate}
\end{theorem}
Setting $x=a$ in the first statement of Theorem~\eqref{th: jjj}, we recover Theorem $2$ of~\cite{Ekilic}.\\
The follwoing result is an extension of Theorem $5$ of \cite{Qyan}.
%%%%%%%%%%%%%%%%%%%%%%%%%%%%%%%%%%%%%%%%%%%%%%%%%%%%%%%%%%%%%%%%%%%%%
\begin{theorem}\label{thm: hhh}
Let $n$ be a positive integer. For $m=1,2,\ldots,n$, we have
\begin{equation*}
\sum_{j=0}^{n}(-1)^{j}{n \brack j}{n+j \brack j}q^{\binom{j}{2}-jn}(\mathcal{H}_{m+j,q}(x)-\mathcal{H}_{j,q}(x))=(-1)^{n}q^{\binom{n+1}{2}}\sum_{i=1}^{m}q^{i}\frac{(xq^{i-n};q)_{n}}{(xq^{i};q)_{n+1}}.
\end{equation*}
where
\begin{equation*}
\mathcal{H}_{n,q}(x)=\sum_{i=1}^{n}\frac{q^{i}}{1-xq^{i}}
\end{equation*}
\end{theorem}
\begin{proof}
By the first identity of Theorem~\eqref{th: jjj}, it is easy to check
\begin{equation*}
\sum_{j=0}^{n}(-1)^{j}{n \brack j}{n+j \brack j}q^{\binom{j}{2}-jn}(\mathcal{H}_{m+j,q}(x)-\mathcal{H}_{j,q}(x))=\sum_{i=1}^{m}q^{i}\frac{(xq^{i})^{n}(q/xq^{i};q)_{n}}{(xq^{i};q)_{n+1}}.
\end{equation*}
By means of
\begin{equation*}
(q/xq^{i};q)_{n}=(-1)^{n}q^{n-in}x^{-n}q^{\binom{n}{2}}(xq^{i-n};q)_{n}
\end{equation*}
we deduce the result.
\end{proof}
%%%%%%%%%%%%%%%%%%%%%%%%%%%%%%%%%%%%%%%%%%%%%%%%%%%%%%%%%%%%%%%%%%%%% 
\section*{Declarations}
\subsection*{Ethical Approval}  
Not applicable.
\subsection*{Competing interests} 
No potential conflict of interest was reported by the authors.
\subsection*{Authors' contributions} 
The authors contributed equally.
\subsection*{Funding} 
This research received no funding.
\subsection*{Availability of data and materials}
Not applicable.

%%%%%%%%%%%%%%%%%%%%%%%%%%%%%%%%%%%%%%%%%%%%%%%%%%%%%%%%%%%%%%%%%%%%%


\begin{thebibliography}{5}
\bibitem{Sahlgren} S. Ahlgren, K. Ono. A Gaussian hypergeometric series evaluation and Ap\'{e}ry number
congruences, J. Reine Angew. Math. 518, 187-212 (2000).
\bibitem{Halzer} H. Alzer and R. Chapman. On Boole's formula for factorials. Australas. J. Combinatorics, 99, 333-336, (2014).
\bibitem{Rapery} R. Ap\'{e}ry. Irrationalit\'{e} de $\zeta(2)$ et $\zeta(3)$ (French), Ast\'{e}risque 61 (1979), 11–13. Luminy Conference
on Arithmetic.
\bibitem{Beukers} F. Beukers. Another congruence for Ap\'{e}ry numbers, J. Number Theory 25, 201-210 (1987).
\bibitem{Wchu} W. Chu. A binomial coefficient identity associated with Beukers’ conjecture on Ap\'{e}ry numbers, Electron. J. Comb. 11 (2004).
\bibitem{Wchu7} W. Chu and Y. You. Binomial symmetries inspired by Bruckman’s problem. Filomat 24, 41–46 (2010).

\bibitem{Comtet} L. Comtet, Advanced combinatorics: the art of finite and infinite expansions. Reidel, Dordrecht 1974.
\bibitem{Ggasper} G. Gasper, M. Rahman. Basic hypergeometric series, Cambridge University Press, Cambridge, 1990.
\bibitem{Igessel} I. Gessel. Some congruences for Ap\'{e}ry numbers, J. Number Theory 14, 3, 362–368 (1982).
\bibitem{Hwgould2} H. W. Gould. Euler's formula for the nth differences of powers, Am. Math. Mon. 85, 450-467 (1978).

\bibitem{Eakaratsuba} E. A. Karatsuba. On an identity with binomial coefficients. Mathematical Notes, 105, 145-147, (2019).
\bibitem{Ekilic} E. Kili\c{c}, H. Prodinger. Evaluation of sums involving products of Gaussian q-binomial coefficients with applications to Fibonomial sums. Turj. J. Math. 41, 707-716 (2017).
\bibitem{Ckrattenthaler} C. Krattenthaler, T. Rivoal, and W. Zudilin. S\'{e}ries hyperg\'{e}om\'{e}triques
basiques, q-analogues des valeurs de la fonction z\^{e}ta et s\'{e}ries d’Eisenstein (French, with
English and French summaries), J. Inst. Math. Jussieu 5 (2006), 1, 53–79.
\bibitem{Mmous} M. Mou\c{c}ouf, S. Zriaa, A new approach for computing the inverse of confluent Vandermonde matrices via Taylor's expansion, Linear Multilinear Algebra. (2021), DOI:10.1080/03081087.2021.1940807.

\bibitem{Cphoata} C. Phoata. Boole's formula as a consequence of Lagrange's interpolating polynomial theorem. Integers, 8, Article A23, (2008).

\bibitem{Jxsharon} J. X. Sharon, J. Zeng. A q-analog of dual sequences with applications. European Journal of Combinatorics, 28, 214-227, (2007).
\bibitem{Astraub} A. Straub. Supercongruences for polynomial analogs of the Ap\'{e}ry numbers. Proc. Am. Math. Soc. 147,
1023–1036 (2019).
\bibitem{Xwang} X. Wang, W. Chu. Harmonic number sums and q-analogues. International Journal of Computer Mathematics: Computer Systems Theory, 4(1), 48-56 (2019).
\bibitem{Qyan} Q. Yan, C. Wei, X. Fan. q-generalizations of Mortenson's identities and further identities. Ramanujan J. 35, 131-139, (2014).
\bibitem{Dyzheng} D. Y. Zheng. An algebraic identity on q-Ap\'{e}ry numbers. Discrete. Math, 311, 23-24 (2011).
\end{thebibliography}
\end{document}